\documentclass{article}
\usepackage[utf8]{inputenc}

\usepackage{amsfonts}
\usepackage{amsmath}
\usepackage{amssymb}
\usepackage{amsthm}

\newtheorem{definition}{Definition}
\newtheorem{lemma}[definition]{Lemma}
\newtheorem{theorem}[definition]{Theorem}
\newtheorem{proposition}[definition]{Proposition}
\newtheorem{remark}[definition]{Remark}
\newtheorem{conjecture}[definition]{Conjecture}
\newtheorem{problem}[definition]{Problem}

\title{Finite Representation Property for Relation Algebra Reducts}
\author{Ja{\v s} {\v S}emrl}
\date{November 2021}

\begin{document}

\maketitle

\begin{abstract}
    The decision problem of membership in the Representation Class of Relation Algebras (RRA) for finite structures is undecidable. However, this does not hold for many Relation Algebra reduct languages. Two well known properties that are sufficient for decidability are the Finite Axiomatisability (FA) of the representation class and the Finite Representation Property (FRP). Furthermore, neither of the properties is stronger that the other, and thus, neither is also a necessary condition. Although many results are known in the area of FA, the FRP remains unknown for the majority of the reduct languages. Here we conjecture that the FRP fails for a Relation Algebra reduct if and only if it contains both composition and negation, or both composition and meet. We then show the right-to-left implication of the conjecture holds and present preliminary results that suggest the left-to-right implication.
\end{abstract}

\section{Introduction}
\label{sec:intro}
Relation Algebra and its relational semantics are a neat algebraic tool for reasoning about various concepts, including program behaviour and correctness \cite{desharnais2009domain} as well as temporal and spatial reasoning \cite{duntsch2005relation}.

This is why the tractability of decision problems in the area is vital for the feasibility of these approaches. Two properties that guarantee the decidability of membership in the representation class are its finite axiomatisability and the finite representation property.

Neither of these properties hold for the full signature of Relation Algebras, in fact, the decision problem itself is undecidable \cite{hirsch2002relation}. This is why we drop some operations in the signature, to obtain a reduct language, and look for better behaviour there.

A number of results regarding finite axiomatisability of Relation Algebra reduct languages exist, though, less is known about the finite representation property. It is known, however, that neither property is stronger than the other. For example, the meet-lattice semigroups $\{\cdot, ;\}$ are a known example of a finitely axiomatisable representation class with no finite representation property \cite{bredihin1978representations, neuzerling2016undecidability}, whereas the demonically refined semigroups have been recently shown to be the first example with composition that has the finite representation property, but a non-finitely axiomatisable representation class \cite{hirsch2021finite}.

Here we examine the finite representation property in detail. We first conjecture that a Relation Algebra reduct language fails to have the finite representation property if and only if it includes both composition and negation, or both composition and meet.

Then we show that if negation and composition are in the signature, the finite representation property fails. This result, together with \cite{neuzerling2016undecidability}, shows the left-to-right implication of the conjecture. We continue by showing that a finite structure in any other reduct signature is finitely representable if and only if it embeds into a finite Relation Algebra. Although this does not prove the right-to-left implication of the conjecture outright, it does suggest that well known counterexamples to the finite representation property have a finite representation in these signatures.

\section{Preliminaries}
We now define the concepts discussed in the introduction formally. We do, however, assume that the reader is familiar with first order logic and Boolean Algebra.

We begin by defining relational composition ($;$), converse ($\smile$), and identity ($1'$). Let $R,S \subseteq X \times X$ be a pair of binary relations over the base $X$. We say
\begin{align*}
    R;S &= \{(x,z) \mid \exists y : (x,y) \in R, (y,z) \in S\}\\
    \Breve{R} &= \{ (y,x) \mid (x,y) \in R\} \\
    1' &= \{(x,x) \mid x \in X\}
\end{align*}
Together with the Boolean Operations, these define the class of proper relation algebras as follows
\begin{definition}[Proper Relation Algebra (PRA)]
The class of Proper Relation Algebras is the class of all $\{0,1,-,+,\cdot,1',\smile,;\}$-structures $\mathcal{S}$ that have some base set $X$, such that if $a \in \mathcal{S}$ then $a \subseteq X \times X$, $\{0,1,-,+,\cdot\}$ are interpreted as proper Boolean operations, and $\{1', \smile, ;\}$ are interpreted as proper relational identity, converse, and composition.
\end{definition}
This gives rise to the class of Representable Relation Algebras (RRA), defined as
\begin{definition}[Representable Relation Algebra (RRA)]
The class of Representable Relation Algebras is the class of Proper Relation Algebras, closed under isomorphic copies.
\end{definition}
\begin{remark}
Much like with the language of Relation Algebras, we may consider looking at relation algebra reduct languages $\tau \subseteq \{0,1,-,+,\cdot, 1', \smile, ;\}$, the class of all proper $\tau$-structures $P(\tau)$, and the class of all representable $\tau$-structures $R(\tau)$.
\end{remark}
\begin{definition}[(Finite) Representation]
A representation is an isomorphism that maps a structure in the representation class to a proper structure. It is finite if the elements of the image are relations over a finite base $X$.
\end{definition}
\begin{definition}[Finite Representation Property (FRP)]
A representation class is said to have the finite representation property if all its finite members have a finite representation.
\end{definition}
\begin{definition}[Finite Axiomatisability (FA)]
The representation class is finitely axiomatisable if membership can be axiomatised by finitely many first order formulas.
\end{definition}
It is known that the language of Relation Algebras has neither FA, nor the FRP. However, it is the case that the equational theory of Relation Algebras is finitely axiomatisable. This gives rise to the class of Relation Algebras (RA), which are all the $\{0,1,-,+,\cdot,1',\smile,;\}$-structures that model the equational theory of representable relation algebras. Alternatively, we can axiomatise the class as follows
\begin{enumerate}
    \item $\{0,1,-,+,\cdot\}$ is a Boolean Algebra
    \item $;, \smile$ are additive over $+$ (and thus monotone over $\leq$)
    \item $-(\Breve{a}) = (- a)^\smile$ and $a;1 = 1 \vee (-a);1 = 1$
    \item $(a;b)^\smile = \Breve{b};\Breve{a}$, $\Breve{\Breve{a}} = a$ , and $\Breve{1'} = 1'$
    \item $a \cdot (b;c) = 0 \Longleftrightarrow b \cdot (a;\Breve{c}) = 0$
\end{enumerate}
It is easy to see that $PRA \subseteq RRA \subseteq RA$. Furthermore, it is known that all of the inclusions are proper and thus $PRA \subset RRA \subset RA$.

Finally, we define the concept of an atom. As Relation Algebras are extensions of Boolean Algebras, we know that they will inevitably have atoms, which are all elements $a \neq 0$ such that for all $b$ if $b \leq a$ it is true that $b = 0 \vee b = a$. Thus every element of a relation algebra can be expressed as a unique join of atoms.

\section{The Conjecture and Related Work}
\label{sec:conj}
In this section we conjecture what reduct signatures of the language of Relation Algebras have the finite representation property. Then we survey the known results in the area and put the findings presented in the later sections into context. We begin by stating our conjecture regarding the FRP for Relation Algebra reduct languages.
\begin{conjecture}
\label{conj:theconjecture}
Let $\tau$ be a Relation Algebra reduct signature. $\tau$ has the finite representation property if and only if
$$\{-,;\} \not \subseteq \tau \not \supseteq \{\cdot,;\}$$
\end{conjecture}
Let us start by examining the known results in the area of the finite representation property. We summarise the results in the table below.
\begin{center}
\begin{tabular}{|c|c|}
    \hline
    \textbf{No FRP} & \textbf{FRP}  \\ \hline
    $\{\cdot, ;\} \subseteq \tau$ \cite{neuzerling2016undecidability} & $\{;\}, \{1', ;\}$ \\
    $\{-, ;\} \subseteq \tau$ [Section~\ref{sec:negcomp}]& $\{\leq,;\}$ \cite{zareckii1959representation} \\
    & $\{D,\smile, ;\} \subseteq \tau \not \supseteq \{-, +, \cdot\}$ \cite{hirsch2013meet,semrl21domain}\\
    & $\{\leq, \setminus, /, \smile, ;\}$ \cite{rogozin2020finite} \\
    & $\{\sqsubseteq, ;\}$ \cite{hirsch2021finite} \\\hline
    
\end{tabular}
\end{center}
We see that the left-to-right implication of Conjecture~\ref{conj:theconjecture} is proven, when we combine our new result from Section~\ref{sec:negcomp} with the result from \cite{neuzerling2016undecidability}. The results on the right-to-left implication are, however, more sparse.

Representable $\{;\}$- and $\{1', ;\}$-structures are known to be finitely representable using the Cayley representation for groups. This result is extended by \cite{zareckii1959representation}, to allow for inclusion of partial ordering (provided there is no relational identity in the signature). In \cite{hirsch2013meet}, an explicit representation of representable $\{0,1,\leq,D,R,1',\smile,;\}$-structures is defined, and in \cite{semrl21domain} we show that this construction works for a wider similarity class of signatures. \cite{rogozin2020finite} defines finite representations for ordered residuated semigroups by embedding them into relational quantales. We also mention our FRP result for demonically refined semigroups where we define an explicit construction of a finite representation in \cite{hirsch2021finite}. Although strictly not a Relation Algebra reduct language, it was the first signature containing composition where the FA fails, but the FRP holds.

Similarly, there exist finitely signatures with both FA and FRP, neither FA nor FRP, and FA but no FRP. Examples are summarised in the table below.

\begin{center}
    \begin{tabular}{|c|cc|}
        \hline & \textbf{FRP} & \textbf{No FRP} \\ \hline
        \textbf{FA} & $\{\leq, ;\}$\cite{zareckii1959representation} & $\{\cdot, ; \}$ \cite{bredihin1978representations,neuzerling2016undecidability}\\
        \textbf{No FA} & $\{\sqsubseteq, ;\}$ \cite{hirsch2021finite} & $\{+,\cdot, ; \}$ \cite{hirsch2012undecidability}\\\hline
    \end{tabular}
\end{center}

\section{Failure of FRP for Negation and Composition}
\label{sec:negcomp}
In this section we prove that if a RA reduct language contains composition and negation, the finite representation property will fail. We do that by showing that the Point Algebra does not have a mapping to a proper relational structure over a finite base that correctly preserves negation and composition. This result, together with that in \cite{neuzerling2016undecidability} shows the FRP fails in all cases where Conjecture~\ref{conj:theconjecture} suggests it will fail, showing its right-to-left implication.

We begin by looking at the Point Algebra and define it as follows
\begin{definition}[Point Algebra]
Point Algebra $\mathcal{P}$ is a Proper Relation Algebra over the base of $\mathbb{Q}$ with eight elements
$$\{0,1,=,\neq, <,\leq, >, \geq \}$$
where $0$ is the empty relation and $1 = \mathbb{Q}\times\mathbb{Q}$ and the rest of the elements are the arithmetic binary predicates defined for $\mathbb{Q}$. Observe that The Point Algebra is closed under all the operations in the language of RA.
\end{definition}
\begin{lemma}
\label{lem:negcompxy}
For any mapping $\theta: \mathcal{P} \rightarrow \wp(X \times X)$ such that $(-a)^\theta = -(a^\theta)$ and $(a;b)^\theta = a^\theta;b^\theta$, there will exist some $x,y \in X$ such that $(x,y) \in 1^\theta$.
\end{lemma}

\begin{proof}
As $-$ is represented correctly and $0 = -1$, it must hold that $1^\theta$ is not an empty relation (i.e. there exists some $(x,y) \in 1^\theta$) or that $0^\theta$ is not an empty relation (i.e. there must exist some $(x,z) \in 0^\theta$). In the latter case observe that since $\theta$ represents composition correctly and $0 = 1;0$, there must exist some $y$ such that $(x,y) \in 1^\theta$.
\end{proof}

\begin{lemma}
\label{lem:negcompxx}
For any mapping $\theta: \mathcal{P} \rightarrow \wp(X \times X)$ such that $|X| < \omega$, $(-a)^\theta = -(a^\theta)$, and $(a;b)^\theta = a^\theta;b^\theta$, there will exist some $x \in X$ such that $(x,x) \in 1^\theta$.
\end{lemma}

\begin{proof}
By Lemma~\ref{lem:negcompxy}, there must exist some $(y,z) \in 1^\theta$. Observe that $1 = 1;1 = ... = 1^{|X|+2} = ...$, so, in order for $\theta$ to represent $;$ correctly, there must exist $x_1, x_2, ..., x_{|X|+1} \in X$ such that $\{(y,x_1), (x_1, x_2), (x_2, x_3), ..., (x_{|X|}, x_{|X|+1}),$ $(x_{|X|+1},z)\} \subseteq 1^\theta$.

As $|X| < \omega$, we know that there exits some $i<j<|X|+2$ such that $x_i = x_j = x$. Observe that since $\{(x_i, x_{i+1}) , (x_{i+1}, x_{i+2}), ..., (x_{j-1},x_j)\} \subseteq 1^\theta$, and $1 = 1^{j-i}$, it must also hold that $(x,x) \in 1^\theta$, as $\theta$ represents $;$ correctly.
\end{proof}

\begin{lemma}
\label{lem:negcompind}
For any mapping $\theta: \mathcal{P} \rightarrow \wp(X \times X)$ such that there exist some $x \in X$ such that $(x_0,x_0) \in 1^\theta$, $(-a)^\theta = -(a^\theta)$, and $(a;b)^\theta = a^\theta;b^\theta$, there must exist, for any $n < \omega$, $x_1, x_2, ..., x_n \in X$ such that $x_i \neq x_j, (x_i, x_j) \in \leq^\theta$, and $ (x_j,x_i) \in >^\theta$, for all $0 \leq i<j \leq n$.
\end{lemma}

\begin{proof}
We show this by induction. In the base case $n=1$. Since $(x_0,x_0) \in 1^\theta$, $1 = \leq;>$, and $\theta$ represents $;$ correctly, we know that there must exist $x_1$ such that $(x_0, x_1) \in \leq^\theta, (x_1, x_0) \in >^\theta$. If it were true that $x_1 = x_0$, then both $(x_0, x_0) \in \leq^\theta$ and $(x_0, x_0) \in >^\theta$ and since $> = -\leq$, $\theta$ would not represent $-$ correctly.

Now assume we have $x_0,x_1, ..., x_n$ for some $0 < n < \omega$ points for which the induction hypothesis holds. Observe that since $(x_n, x_0) \in >^\theta, (x_0,x_n) \in \leq^\theta$, and $\theta$ represents composition, it must hold that $(x_n,x_n) \in (>;\leq)^\theta = 1^\theta$. Thus, again to represent $;$, there must exist $x_{n+1}$ such that $(x_{n}, x_{n+1}) \in \leq^\theta, (x_{n+1}, x_{n}) \in >^\theta$, and as $>;> = >$ and $\leq;\leq = \leq$ it also holds that for all $i<n: (x_{i}, x_{n+1}) \in \leq^\theta, (x_{n+1}, x_{i}) \in >^\theta$. Observe that if $x_{n+1}$ was the same as $x_i$ for some $i \leq n$, it would hold that $(x_{i}, x_{n}) = (x_{n+1}, x_{n}) \in \leq^\theta$ as well as in $>^\theta$. Thus the induction hypothesis is preserved.
\end{proof}

\begin{theorem}
Any RA-reduct signature containing $\{-, ;\}$ fails to have the FRP.
\end{theorem}

\begin{proof}
Suppose a signature $\{-,;\} \subseteq \tau$ had the FRP. Then the $\tau$-reduct of the Point Algebra would have a representation $\theta$ over a finite base $X$. By Lemma~\ref{lem:negcompxx}, there must exist some $x \in X$ such that $(x,x) \in 1^\theta$. This, together with Lemma~\ref{lem:negcompind} is sufficient for $X$ to have $|X|+1$ distinct points and we have reached a contradiction.
\end{proof}

\section{Embedding into RA and FRP}
In this section we examine why the Point Algebra, a counterexample to the FRP for signatures above $\{\cdot,;\}$ and $\{-,;\}$, has a finite representation in all other Relation Algebra reduct signatures. In fact, any structure in a signature not above $\{\cdot,;\}$ nor $\{-,;\}$ that embeds into a finite relation algebra (not necessarily representable) will have a finite representation.

Take a relation algebra $\mathcal{S}$ with the set of atoms $A$ and define a mapping 

$$\theta: \mathcal{S} \rightarrow \wp(A \times A)$$
where
$$(a,b) \in \mathcal{S}^\theta \Longleftrightarrow a;s \geq b$$

and observe the following

\begin{lemma}
\label{lem:embeds}
For any $s \not \leq t \in \mathcal{S}$, there exists a pair $(a,b) \in s^\theta - t^\theta$.
\end{lemma}

\begin{proof}
Observe that since $s \not \leq t$, there must exist an atom $a \leq s-t$. If $a$ is an atom, so is $1' \cdot (a;\Breve{a}) = D(a)$. Thus $D(a);s \geq a$, by monotonicity, but not $D(a);t \geq a$ and thus $(D(a),a) \in s^\theta - t^\theta$.
\end{proof}

\begin{lemma}
$\theta$ represents $;, 1', \smile$ correctly.
\end{lemma}

\begin{proof}
Observe how by associativity and monotonicity if $a;s \geq b$ and $b;t \geq c$, we have $a;s;t \geq c$, so $s^\theta;t^\theta \leq (s;t)^\theta$. For $s^\theta;t^\theta \geq (s;t)^\theta$, observe that if $a;(s;t) \geq b$, then by associativity $(a;s);t \geq b$. By additivity and distributivity, there must thus exist an atom $c \leq a;s$ such that $c;t \geq b$.

Observe that for all $a$ it is true by the identity law and monotonicity that $a;1' \geq a$. Furthermore, if $a,b$ are atoms such that $a;1' = a \geq b$ then it must be the case that $a = b$.

Finally, if $a;s \geq b$, then $a;s \cdot b \neq 0$, and, by Percian triangle law, $b;\Breve{s} \cdot a \neq 0$. Since $a$ is an atom, this implies $a \leq b;\Breve{c}$ and thus it holds that if $(a,b) \in s^\theta$ then $(b,a) \in (\Breve{s})^\theta$.
\end{proof}

\begin{lemma}
\label{lem:embedf}
$0, 1, +$ are correctly represented by $\theta$.
\end{lemma}

\begin{proof}
As $a;0 = 0$, there is no pair of atoms $a,b$ such that $a;0 \geq b$ and thus $0$ is represented correctly. Furthermore, if there exists some $s$ such that $a;s \geq b$, then $a;1 \geq b$, by monotonicity, so $1$ is represented correctly as well.

Let us show that $s^\theta + t^\theta \leq (s+t)^\theta$. Without loss, suppose $a;s \geq b$. Then $a;(s+t) \geq b$ as well, by monotonicity. Finally, if $a;(s+t) \geq b$, then by distribuitivity of the lattice, it must hold that $a;s \geq b$ or $a;t \geq b$, so $s^\theta + t^\theta \geq (s+t)^\theta$.
\end{proof}

Thus we can conclude the following

\begin{proposition}
\label{prop:embed}
If $\tau$ is above neither $\{-,;\}$ nor $\{\cdot, ;\}$ and a finite $\tau$-structure embeds into a finite relation algebra, then it is finitely representable.
\end{proposition}

\begin{proof}
By Lemmas~\ref{lem:embeds}-\ref{lem:embedf}, $\theta$ is a finite representation for a $\tau$-reduct of a Relation Algebra and all its substructures.
\end{proof}
Although this result does not prove the right-to-left implication of Conjecture~\ref{conj:theconjecture} outright, it does suggest that finding a counterexample to the FRP for any of the signatures, conjectured to have the FRP, is going to be a difficult undertaking. Point Algebra, or the Anti-Monk Algebra, two well known counterexamples to the FRP in certain signatures are both Relation Algebras and, as such, have a finite representation in signatures that do not contain both negation and composition, nor both meet and composition.

However, there do exist signatures where it is known that not all finite representable structures embed into a finite Relation Algebra. It is important to add, though, that these proofs heavily rely on the signature in question to be above $\{+,\cdot,;\}$ \cite{hirsch2012undecidability} or $\{-, ;\}$ \cite{neuzerling2016undecidability}.

\section{Problems}
Finally, we look at some open problems in the area of the Finite Representation Property. We have shown that the left-to-right implication of Conjecture~\ref{conj:theconjecture} holds and provided a proposition that suggests the right-to-left implication may hold as well. However, it does remain an open question whether or not these signatures have the FRP.
\begin{problem}
Do signatures $\tau$ that are above neither $\{\cdot, ;\}$ nor $\{-,;\}$ have the Finite Representation Property? 
\end{problem}
Answering this question, as a result of Proposition~\ref{prop:embed}, is equivalent to the following
\begin{problem}
For signatures $\tau$ above neither $\{\cdot, ;\}$ nor $\{-,;\}$, does every finite representable structure embed into a $\tau$-reduct of a finite Relation Algebra?
\end{problem}
However, defining such an embedding is not a trivial undertaking, as closing the structure of meet-negation completions under an associative composition operation gives rise to a number of challenges.

As discussed in Section~\ref{sec:intro}, FA and FRP are both sufficient conditions for the decidability of the membership in the representation class for finite structures. However, the following is not known.

\begin{problem}
Does there exist a Relation Algebra reduct language with decidable membership in representation class for finite structures decision problem, but neither FA nor FRP?
\end{problem}

A related decision problem remains open for the Language of Relation algebras and a number of its reducts.

\begin{problem}[\cite{maddux1994perspective}]
Is determining whether a finite Relation Algebra has a finite representation decidable?
\end{problem}

We have touched on the Finite Representation Property of demonically refined semigroups in Section~\ref{sec:conj}. The demonic refinement predicate is part of the Demonic Relational Calculus, defined to model the behaviour of nondeterministic programs when the Demon is in control of the nondeterministic choice. This gives rise to the following problems.
\begin{problem}
Does every operation in the language of Relation Algebra have a demonic coutnerpart? Does this allow us to define the Demonic Relation Algebra?
\end{problem}
\begin{problem}
What Demonic Relation Algebra Reduct Signatures have FA representation class? Which have the FRP? What about mixed signatures?
\end{problem}

\bibliographystyle{alpha}
\bibliography{ref}

\end{document}